\documentclass[options]{asl}

\usepackage{amssymb}
\usepackage[final]{showkeys}
 \usepackage{enumitem}
 \usepackage{tikz -inet}
\usetikzlibrary{positioning,shapes.misc, patterns,arrows,decorations,decorations.pathreplacing, snakes}

\newtheorem{theorem}{Theorem}[section]
\newtheorem*{theorem1}{Theorem 1.2}
\newtheorem{assumption}[theorem]{Assumption}
\newtheorem{definition}[theorem]{Definition}
\newtheorem{proposition}[theorem]{Proposition}
\newtheorem{lemma}[theorem]{Lemma}
\newtheorem{remark}[theorem]{Remark}

\newtheorem{conjecture}[theorem]{Conjecture}
\newtheorem{fact}[theorem]{Fact}

\newcommand{\K}{\mathcal{K}}
\newcommand{\Union}{\bigcup}
\DeclareMathOperator{\tp}{ga-tp}

\DeclareMathOperator{\Aut}{Aut}
\DeclareMathOperator{\cf}{cf}
\newcommand{\gaS}{\operatorname{ga-S}}
\DeclareMathOperator{\LS}{LS}
\newcommand{\T}{\mathcal{T}}

\newcommand{\C}{\mathfrak{C}}
\newcommand{\St}{\operatorname{\mathfrak{S}t}}

\newcommand{\preck}{\prec_\K}

\def\runiv{\rotatebox[origin=c]{-90}{$\prec_{u}$}} 

\newcommand{\footnotei}[1]{}

\newcommand{\comment}[1]{}

\newcommand{\rest}{\upharpoonright}

\title[Limit Models in Strictly Stable AECs]
{Limit Models in Strictly Stable Abstract Elementary Classes}


\author{Will Boney}

\address{Mathematics Department\\
Texas State University\\
San Marcos, TX 78666}

\email{wb1011@txstate.edu}

\author{Monica M. VanDieren}

\thanks{This material is based upon work while the first author was supported by the National Science Foundation under Grant No. DMS-1402191 and DMS-2137465.}

\date{\today}

\begin{document}

\begin{abstract}

In this paper, we examine the locality condition for non-splitting and determine the level of uniqueness of limit models that can be recovered in some stable, but not superstable, abstract elementary classes.  In particular we prove:

\begin{theorem1}
Suppose that $\K$ is an abstract elementary class satisfying
\begin{enumerate}
\item the joint embedding and amalgamation properties with no maximal model of cardinality $\mu$. 
\item  stability in $\mu$.
\item $\kappa^*_\mu(\K)<\mu^+$.
\item  continuity for non-$\mu$-splitting (i.e. if $p\in\gaS(M)$ and $M$ is a limit model witnessed by $\langle M_i\mid i<\alpha\rangle$ for some limit ordinal $\alpha<\mu^+$ and there exists $N \prec M_0$ so that $p\restriction M_i$ does not $\mu$-split over $N$ for all $i<\alpha$, then $p$ does not $\mu$-split over $N$).
\end{enumerate}
 For $\theta$ and $\delta$ limit ordinals $<\mu^+$ both with cofinality $\geq\kappa^*_\mu(\K)$,
if $\K$ satisfies symmetry for non-$\mu$-splitting (or just $(\mu,\delta)$-symmetry\footnotei{Track down exactly what symmetry is needed; the referee things that `$(\mu, \lambda)$-symmetry for some $\lambda$' should suffice.}), then, for any $M_1$ and $M_2$ that are $(\mu,\theta)$ and $(\mu,\delta)$-limit models over $M_0$, respectively, we have that $M_1$ and $M_2$ are isomorphic over $M_0$.

\end{theorem1}

Note that no tameness is assumed.

\end{abstract}

\maketitle

\section{Introduction}
Because the main test question for developing a classification theory for abstract elementary classes (AECs) is Shelah's Categoricity Conjecture \cite[Problem D.1]{Ba},
 the development of independence notions for AECs has often started with an assumption of categoricity (\cite{sh576, vas3, V1} and others). Consequently, the independence relations that result are superstable or stronger (see, for instance, good $\lambda$-frames and the superstable prototype \cite[Example II.3.(A)]{shelahaecbook}).  However, little progress has been made to understand stable, but not superstable AECs.  
 A notable exception is the work on $\kappa$-coheir of Boney and Grossberg \cite{bgcoheir}, which only requires stability in the guise of `no weak $\kappa$-order property.'  In this paper, we  add to the understanding of strictly stable AECs with a different approach and under different assumptions than \cite{bgcoheir}.  In particular, our analysis uses towers and the standard definition of Galois-stability.  Moreover, we work without assuming any of the strong locality assumptions (tameness, type shortness, etc.) of \cite{bgcoheir}.  We hope that this work will lead to further exploration in this context.

The main tool in our analysis is a tower, which was first conceived to study superstable AECs (see, for instance \cite{ShVi} or \cite{Va1}).
The `right analogue' of superstability in AECs has been the subject of much research.  Shelah has commented that this notion suffers from `schizophrenia,' where several equivalent concepts in first-order seem to bifurcate into distinct notions in nonelementary settings; see the recent Grossberg and Vasey \cite{grva} for a discussion of the different possibilities (and a suprising proof that they are equivalent under tameness).  

Common to much analysis of superstable AECs is the uniqueness of limit models.  Uniqueness of limit models was first proved to follow from a categoricity assumption in \cite{sh394, Sh:600, ShVi, Va1, Va1e}.  Later,  $\mu$-superstability, which was isolated by Grossberg, VanDieren, and Villaveces \cite[Assumption 2.8(4)]{gvv}, was shown to imply uniqueness of limit models under the additional assumption of $\mu$-symmetry \cite{Va2}.  $\mu$-superstability was modeled on the local character characterization of superstability in first-order and was already known to follow from categoricity \cite{ShVi}.  
The connection between $\mu$-symmetry and structural properties of towers \cite{Va2} 
inspired  recent research on $\mu$-superstable classes: \cite{Va4, VaVa}.  Moreover, years of work culminating in the series of papers \cite{ShVi, Va1, Va1e, Va2, Va4, VaVa} has led to the extraction of a general scheme for proving the uniqueness of limit models (note that amalgamation is generally assumed in these papers, but this is not true of \cite{ShVi, Va1, Va1e}).  
In this paper we witness the power of this new scheme by adapting the technology developed in \cite{Va2} to cover $\mu$-stable, but not $\mu$-superstable classes.  We suspect that this new technology of towers will likely be used to answer other problems in classification theory (in both first order and non-elementary settings).  

This paper focuses on the question to what degree the uniqueness of limit models can be recovered if we assume the class is Galois-stable in $\mu$, but not $\mu$-superstable, by refocusing the question from ``\emph{Are all} $(\mu, \alpha)$-limit models isomorphic (over the base)?'' to ``\emph{For which $\alpha, \beta < \mu^+$ are } $(\mu, \alpha)$-limit models and $(\mu, \beta)$-limit models isomorphic (over the base)?''  Based on first-order results (summarized in \cite[Section 2]{gvv}), we have the following conjecture; note that $\mu$-stability implies that limit models exist so the set below is meaningful.

\begin{conjecture}\label{stab-conj}
Suppose $\K$ is an AEC with $\mu$-amalgamation and is $\mu$-stable.  The set 
$$\left\{ \alpha < \mu^+ : \cf(\alpha)=\alpha\text{ and } (\mu, \alpha)\text{-limit models are isomorphic to }(\mu, \mu)\text{-limit models} \right\}$$
is a non-trivial interval of regular cardinals.  Moreover, the minimum of this set is an important measure of complexity of $\K$, namely it is is $\kappa^*_\mu(\K)$ (see Definition \ref{kappastar-def}).
\end{conjecture}

Our main result (restated from the abstract)\footnotei{WB: The theorem1 environment is now hardcoded to be `Theorem 1.2.' If the label here changes, we should change that environment}  proves this conjecture under certain assumptions.

\begin{theorem}\label{uniqueness theorem}
Suppose that $\K$ is an abstract elementary class satisfying
\begin{enumerate}
\item the joint embedding and amalgamation properties with no maximal model of cardinality $\mu$. 
\item  stabilty in $\mu$.
\item $\kappa^*_\mu(\K)<\mu^+$.
\item  continuity for non-$\mu$-splitting (i.e. given a model $N$, a $(\mu, \alpha)$-limit model $M$ witnessed by $\langle M_i\mid i<\alpha\rangle$, and $p\in\gaS(M)$, if $p\restriction M_i$ does not $\mu$-split over $N$ for all $i<\alpha$, then $p$ does not $\mu$-split over $N$).
\end{enumerate}
 For $\theta$ and $\delta$ limit ordinals $<\mu^+$ both with cofinality $\geq\kappa^*_\mu(\K)$,
if $\K$ satisfies symmetry for non-$\mu$-splitting (or just $(\mu,\delta)$-symmetry), then, for any $M_1$ and $M_2$ that are $(\mu,\theta)$ and $(\mu,\delta)$-limit models over $M_0$, respectively, we have that $M_1$ and $M_2$ are isomorphic over $M_0$.

\end{theorem}

Assumption \ref{assm} collects these assumptions together, and we discuss them following that statement.   In this statement, the ``measure of complexity''  from Conjecture \ref{stab-conj} is $\kappa^*_\mu(\K)$, a generalization of the first-order $\kappa(T)$ (see Definition \ref{kappastar-def}).  An important feature of this work is that it explores the underdeveloped field of strictly stable AECs.  

We end with a short comment contextualizing this paper within the body of work on limit models.  The general arguments for investigating the uniqueness of limit models have appeared before (see \cite{Va1, gvv}). One use is that they give a version of saturated models without dealing with smaller models and give a sense of how difficult it is to create saturated models.  Many works of AECs take a `local approach' of analyzing $\K_\lambda$ (the models of size $\lambda$) to derive structure on $\K_{\lambda^+}$ (see \cite[Chapter II]{shelahaecbook} or \cite{sh576} for the most prominent examples).  Because not even the existence of models of size $<\lambda$ is assumed, Galois saturation (which quantifies over smaller models) cannot be used, and limit models have become the standard substitute. Moreover, we expect that limit models will take on a greater importance in the context of strictly stable AECs, especially those without assumption of tameness.  Of the various analogues for AECs (see \cite[Theorem 1.2]{grva}), most have seen extensive analysis, but only in the context of tameness.  One of the remaining notions (solvability; see \cite[Chapter IV]{shelahaecbook}) seems to have no weakening to the strictly stable context.  What remains are $\mu$-superstability and the uniqueness of limit models.  Thus, it is reasonable to assume that understanding strictly stable AECs will require understanding the connection between `$\mu$-stability' (Assumption \ref{assm} here) and limit models.  Theorem \ref{uniqueness theorem} is a step towards this understanding.  

After circulating this paper but before publication, Vasey and Mazari-Armida used our results to make further progress in the field.  Vasey used Theorem \ref{uniqueness theorem} in his work to characterize stable AECs \cite{v-stab-AEC}, especially in terms of unions of sufficiently saturated models being saturated \cite[Theorem 11.11]{v-stab-AEC}.  Additionally, Vasey \cite[Theorem 3.7]{v-stab-AEC} gives some natural conditions for Assumption \ref{assm}.(\ref{wc-split}) below, which he calls the weak continuity of splitting. 
On the other hand, Mazari-Armida identified naturally occuring strictly stable AECs.  By analyzing limit models of different cofinalities, he demonstrated that the class of torsion-free abelian groups and the class of finitely Butler groups, both with the pure subgroup relation, are strictly stable AECs \cite{marcos}.

Section \ref{background-sec} reviews key definitions and facts with Assumption \ref{assm} being the key hypotheses throughout the paper.  Section \ref{relfulltow-sec} discusses the notion of relatively full towers.  Section \ref{redtow-sec} discusses reduced towers and proves the key lemma, Theorem \ref{reduced are continuous}.  Section \ref{ulm-sec} concludes with a proof of the main theorem, Theorem \ref{uniqueness theorem}.

We would like to thanks Rami Grossberg, Sebastien Vasey, and the referee for comments on earlier drafts of this paper that led to a vast improvement in presentation.

\section{Background} \label{background-sec}
We refer the reader to \cite{Ba}, \cite{GV2}, \cite{gvv}, \cite{Va1}, and \cite{Va2} for definitions and notations of concepts such as Galois-stability, $\mu$-splitting,  etc.  We reproduce a few of the more specialized definitions and results here.

Grossberg, VanDieren, and Villaveces \cite[Assumption 2.8]{gvv} isolated a notion they call `$\mu$-superstability'\footnote{We do not use this here, but the definition of $\mu$-superstability strengthens Assumption \ref{assm} by requiring that $\kappa^*_\mu(\K)$ be $\omega$.} by examining consequences of categoricity from \cite{sh394} and \cite{ShVi}.  The key feature in this assumption is that there are no infinite splitting chains (as forbidden in \cite[Theorem 2.2.1]{ShVi}).  We weaken $\mu$-superstability by only forbidding long enough splitting chains.  How long is `long enough' is measured by $\kappa^*_\mu(\K)$, which is a relative of \cite[Definition 4.3]{GV2} and universal local character \cite[Definition 3.5]{bgcoheir}.  Following \cite{bgcoheir}, we add the `*' to this symbol to denote that the chain is required to have the property that $M_{i+1}$ is universal over $M_i$.

\begin{definition} \label{kappastar-def}
We define $\kappa^*_\mu(\K)$ to be the minimal, regular $\kappa<\mu^+$ so that for every increasing and continuous sequence $\langle M_i\in\K_\mu\mid i\leq\alpha \rangle$ with $\alpha\geq \kappa$ regular which satisfies
 for every $i<\alpha$, $M_{i+1}$ is universal over $M_i$, and for every non-algebraic $p\in\gaS(M_\alpha)$, 
there exists $i<\alpha$ such that $p$ does not $\mu$-split over $M_i$.  If no such $\kappa$ exists, we say $\kappa^*_\mu(\K)=\infty$.

We call $\kappa^*_\mu(\K)$ the `universal local character for $\mu$-nonsplitting for $\K$,' or simply the `universal local character' for short when $\mu$ and $\K$ are fixed.
\end{definition}

In \cite[Theorem 4.13]{GV2}, Grossberg and VanDieren show that  if $\K$ is a tame stable abstract elementary class satisfying the joint embedding and amalgamation properties with no maximal models, then there exists a single bound for $\kappa^*_\mu(\K)$  for all sufficiently large $\mu$ in which $\K$ is $\mu$-stable.  This proof works by considering the $\chi$-order property of Shelah.  We can also give a direct bound assuming tameness.

\begin{proposition}
Let $\K$ be an AEC with amalgamation that is $\lambda$-stable and $(\lambda, \mu)$-tame.  Then $\kappa^*_\mu(\K) \leq \lambda$.
\end{proposition}

Note that the proof does not require the extensions to be universal.

\begin{proof}
Let $\langle M_i \in K_\mu : i \leq \alpha \rangle$ be an increasing, continuous chain with $\cf(\alpha) \geq \lambda$ and $p \in \gaS(M_\alpha)$.  By \cite[Claim 3.3.(1)]{sh394} and $\lambda$-stability, there is $N_0 \prec M_\alpha$ of size $\lambda$ such that $p$ does not $\lambda$-split over $N_0$.  By tameness, $p$ does not $\mu$-split over $N_0$.  By the cofinality assumption, there is $i_* < \alpha$ such that $N_0 \prec M_{i_*}$.  By monotonicity, $p$ does not $\mu$-split over $M_{i_*}$.
\end{proof}

This definition motivates our main assumption.  We use this collection only to group these items together and will explicitly list Assumption \ref{assm} when it is part of a result's hypothesis.

\begin{assumption}\label{assm}
\mbox{}
\begin{enumerate}
\item $\K$ satisfies the joint embedding and amalgamation properties with no maximal model of cardinality $\mu$. 
\item $\K$ is stable in $\mu$.
\item $\kappa^*_\mu(\K)<\mu^+$.
\item \label{wc-split}$\K$ satisfies (limit) continuity for non-$\mu$-splitting (i.e. if $p\in\gaS(M)$ and $M$ is a limit model witnessed by $\langle M_i\mid i<\theta\rangle$ for some limit ordinal $\theta<\mu^+$ and there exists $N$ so that $p\restriction M_i$ does not $\mu$-split over $N$ for all $i<\theta$, then $p$ does not $\mu$-split over $N$).
\end{enumerate}

\end{assumption}

A few comments on the assumption is in order.  Note that tameness is not assumed in this paper.  Amalgamation is commonly assumed in the study of limit models, although \cite{ShVi,Va1, Va1e} replace it with more nuanced results about amalgamation bases. Stability in $\mu$ is necessary for the conclusion of Theorem \ref{uniqueness theorem} to make sense; otherwise, there are no limit models!  We have argued (both in principle and in practice) that varying the local character cardinal is the right generalization of superstability to stability in this context.  However, we have kept the ``continuity cardinal'' to be $\omega$; this is the content of Assumption \ref{assm}.(\ref{wc-split}).  This seems necessary for the arguments\footnote{The first author claimed in the discussion following \cite[Lemma 9.1]{extendingframes} that only long continuity was necessary.  However, after discussion with Sebastien Vasey, this seems to be an error.}.  It seems reasonable to hope that some failure of continuity for non-splitting will lead to a nonstructure result, but this has not yet been achieved.

The assumptions are (trivially) satisfied in any superstable AEC and, therefore, any categorical AEC.  However, in this context, the result is already known.  For a new example, we look to the context of strictly stable homogeneous structures as developed in Hyttinen \cite[Section 1]{hyttinen}.  In the homogeneous contexts, Galois types are determined by syntactic types.  Armed with this, Hyttinen studies the normal syntactic notion of nonsplitting under a stable, unsuperstable hypothesis \cite[Assumption 1.1]{hyttinen}, and shows that syntactic splitting satisfies continuity and (more than) the universal local character of syntactic nonsplitting is $\aleph_1$.\footnote{It shows that it is at most $\aleph_1$.  However, if it were $\aleph_0$, the class would be superstable, contradicting the assumption.}  It is easy to see that the syntactic version of nonsplitting implies our nonsplitting, which already implies $\kappa^*_\mu(\K) = \aleph_1$.  The following argument shows that, if $N$ is limit over $M$, the converse holds as well, which is enough to get the limit continuity for our semantic definition of splitting.  Since the context of homogeneous model theory is very tame, we don't worry about attaching a cardinal to non-splitting because they are all equivalent.

Suppose that $N$ is a limit model over $M$, witnessed by $\langle M_i \mid i < \alpha \rangle$, and $p \in \gaS(M)$ syntactically splits over $M$.  Then, since Galois types are syntactic, there are $b, c \in N$ such that $\tp(b/M) = \tp(c/M)$ and, for an appropriate $\phi$, $\phi(x, b, m) \wedge \neg \phi(x, c, m) \in p$.  We can find $\beta, \beta' < \alpha$ such that $b \in N_\beta$ and $c \in N_{\beta'}$.  Since $b$ and $c$ have the same type, we can find an amalgam $N_* \succ N_\beta$ and $f:N_\alpha \to_{M} N_*$ such that $f(b) = c$.  Since $N$ is universal over $N_{\beta'}$, we can find $h:N_* \to_{N_{\beta'}} N$.  This gives us an isomorphism $h \circ f:N_\beta \cong h(f(N_\beta))$ and we claim that this witnesses the semantic version splitting: $c \in N_{\beta'}$, so $c = h(c) = h(f(b)) \in h(f(N_\beta))$ and, thus, $\neg \phi(x, c, m) \in p \upharpoonright h(f(N_\beta))$.  On the other hand, $\phi(x, c, m) = h\circ f(\phi(x, b, m)) \in h \circ f( p \upharpoonright N_\beta)$.  Thus, we have witnessed $h \circ f(p \upharpoonright N_\beta) \neq p \upharpoonright h(f(N_\beta))$.

Note if $\kappa^*_\mu(\K) = \mu$, then the conclusion of Theorem \ref{uniqueness theorem} is uninteresting, but the results still hold: any two limit models whose lengths have the same cofinality are isomorphic on general grounds.  Also, we assume joint embedding, etc. only in $\K_\mu$. However, to simplify presentation, we work as though these properties held in all of $\K$ and, thus, we work inside a monster model.  This will allow us to write $\tp(a/M)$ rather than $\tp(a/M;N)$ and witness Galois type equality with automorphisms.  The standard technique of working inside of a $(\mu,\mu^+)$-limit model  can translate our proofs to ones not using a monster model.

Under these assumptions, it is possible to construct towers.  This is the key technical tool in this construction.  Towers were introduced in Shelah and Villaveces \cite{ShVi} and expanded upon in \cite{Va1} and subsequent works.

Since $I$ is well-ordered, it has a successor function which we will denote $+1$ (or $+_I 1$ if necessary).  Also, we typically restrict our attention to well-ordered $I$.

\begin{definition}[{\cite[Definition I.5.1]{Va1}}]\

\begin{enumerate}

\item A \emph{tower indexed by $I$ in $\K_\mu$} is a triple $\T = \langle \bar M, \bar a, \bar N \rangle$ where

\begin{itemize}
	\item $\bar M=\langle M_i\in\K_\mu\mid i\in I\rangle$ is an increasing sequence of limit models;
	\item $\bar a=\langle a_{i}\in M_{i+1}\backslash M_i\mid i+1\in I\rangle$ is a sequence of elements; 
	\item $\bar N=\langle N_{i} \in K_\mu \mid i+1\in I\rangle$ such that $N_i \prec M_i$ with $M_i$ universal over $N_i$; and
	\item $\tp(a_i/M_i)$ does not $\mu$-split over $N_i$.
\end{itemize}
\item A tower $\T = \langle \bar M, \bar a, \bar N\rangle$ is \emph{continuous} iff $\bar M$ is, i. e., $M_i = \cup_{j<i} M_j$ for all limit $i \in I$.
\item $\K^*_{\mu, I}$ is the collection of all towers indexed by $I$ in $\K_\mu$.

\end{enumerate}
\end{definition}

Crucially, in the above definition, if $I$ has a maximum element $i$, then a tower $\T = (\bar{M}, \bar{a}, \bar{N})$ has a model $M_i$, but not terms $a_i$ or $N_i$. Note that continuity is not required of all towers.

 We will switch back and forth between the notation $\K^*_{\mu,\alpha}$ where $\alpha$ is an ordinal and $\K^*_{\mu,I}$ where $I$ is a well ordered set (of order type $\alpha$) when it will make the notation clearer.  
When we deal with relatively full towers, we will find the notation using $I$ to be more convenient for book-keeping purposes. 

For $\beta<\alpha$ and $\T=(\bar M,\bar a,\bar N)\in\K^*_{\mu,\alpha}$ we write $\T\restriction \beta$ for the tower made up of the sequences $\bar M\restriction \beta:=\langle M_i\mid i<\beta\rangle$, $\bar a\restriction\beta:=\langle a_i\mid i+1<\beta\rangle$, and $\bar N\restriction \beta:=\langle N_i\mid i+1<\beta\rangle$.
 
 We will construct increasing chains of towers.  Here we define what it means for one tower to extend another:
 
\begin{definition}
For $I$ a sub-ordering of $I'$ and
 towers $(\bar M,\bar a,\bar N)\in\K^*_{\mu,I}$ and $(\bar M',\bar a',\bar N')\in\K^*_{\mu,I'}$, we say $$(\bar M,\bar a,\bar N)\leq (\bar M',\bar a',\bar N')$$ if $\bar a=\bar a'\restriction I$, $\bar N=\bar N'\restriction I$, and for $i\in I$, $M_i\preceq_{\K}M'_i$  and whenever $M'_i$ is a proper extension of $M_i$, then $M'_i$ is universal over $M_i$.  If for each $i\in I$,  $M'_i $ is universal over $M_i$ we will write $(\bar M,\bar a,\bar N)< (\bar M',\bar a',\bar N')$.
 
 Note if $I' = I$, then we have that $\bar{a} = \bar{a}'$ and $\bar{N} = \bar{N}'$.
\end{definition}

For $\gamma$ a limit ordinal  $<\mu^+$ and $\langle I_j\mid j<\gamma\rangle$ a sequence of well ordered sets with $I_j$ a sub-ordering of $I_{j+1}$, if $\langle(\bar M^j,\bar a,\bar N)\in\K^*_{\mu,I_j}\mid j<\gamma\rangle$ is a $<$-increasing sequence of towers, then the union $\T$ of these towers is determined by the following: 
\begin{itemize}
\item for each $\beta\in \Union_{j<\gamma}I_j$,
 $M_\beta:=\Union_{\beta\in I_j;\; j<\gamma}M^j_\beta$ 
 \item the sequence $\langle a_\beta\mid \exists (j<\gamma)\; \beta+1,\beta\in I_j\rangle$, and  
 \item the sequence $\langle N_\beta\mid \exists (j<\gamma)\; \beta+1,\beta\in I_j\rangle$
 \end{itemize}  is a tower in $\K^*_{\mu,\Union_{j<\gamma}I_j}$, provided that $\K$ satisfies the continuity property for non-$\mu$-splitting and that $\Union_{j<\gamma} I_j$ is well ordered.  Note that it is our desire to take increasing unions of towers that leads to the necessity of the continuity property.

We also need to  recall a few facts about directed systems of partial extensions of towers 
 that are implicit in \cite{Va1}.  These are helpful tools in the inductive construction of towers and are used in other work (see, e.g., \cite[Facts 2 and 3]{Va2}):  Fact \ref{successor stage prop} will get us through the successor step of inductive constructions of directed systems, and Fact \ref{limit stage prop} describes how to pass through the limit stages.  An explicit proof of Fact \ref{limit stage prop} appears as \cite[Fact 3]{Va2}, and we provide a proof of Fact \ref{successor stage prop} below.  Two important notes:
 \begin{itemize}
     \item These facts do not require that the towers be continuous.
     \item The work in \cite{Va1} does not assume amalgamation, so more care had to be taken in working with large limit models (in place of the monster model) and towers made of amalgamation bases.  The amalgamation assumption in this (and other) papers significantly simplifies the situation.
 \end{itemize}

\begin{fact}[\cite{Va1}]\label{successor stage prop}
Suppose $\T$ is a tower in $\K^*_{\mu,\alpha}$ and $\T'$ is a tower of length $\beta<\alpha$ with $\T\restriction \beta<\T'$, if $f\in\Aut_{M_\beta}(\C)$ and $M''_\beta$ is a limit model universal over $M_{\beta}$ such that $\tp(a_\beta/M''_\beta)$ does not $\mu$-split over $N_\beta$ and $f(\Union_{i<\beta}M'_i)\prec_{\K}M''_\beta$, then the tower $\T''\in\K^*_{\mu,\beta+1}$ defined by $f(\T')$ concatenated with the model $M''_\beta$, element $a_\beta$ and submodel $N_\beta$ is an extension of $\T\restriction (\beta+1)$.
\end{fact}

\begin{proof}
This is a routine verification from the definitions. $\T''\rest \beta$ is isomorphic to the tower $\T'$ and we are given the required nonsplitting and that, for $i < \beta$, $f(M_i')\prec M_\beta''$, so we have that $\T'' \in \K^*_{\mu, \beta+1}$.  Similarly, $f$ fixes $\T \rest \beta$, so $\T \rest \beta < \T'$ implies $\T \rest \beta < \T'' \rest \beta$.  To extend this to $\T \rest (\beta+1) < \T''\rest (\beta+1) = \T''$, we note that $M_\beta''$ is universal over $M_\beta$ by assumption.
\end{proof}

\begin{fact}[\cite{Va1}]\label{limit stage prop}
Fix $\T\in\K^*_{\mu,\alpha}$ for $\alpha$ a limit ordinal.
Suppose $\langle \T^i\in\K^*_{\mu,i}\mid i<\alpha\rangle$  and $\langle f_{i,j}\mid i\leq j<\alpha\rangle$ form a directed system of towers.  Suppose
\begin{itemize}
\item each $\T^i$ extends $\T\restriction i$
\item $f_{i,j}\restriction M_i=id_{M_i}$
\item $M^{i+1}_{i+1}$ is universal over 
$f_{i,i+1}(M^i_i)$.
\end{itemize}
Then there exists a direct limit $\T^\alpha$ and mappings $\langle f_{i,\alpha}\mid i<\alpha\rangle$ to this system so that $\T^\alpha\in\K^*_{\mu,\alpha}$, $\T^\alpha$ extends $\T$, and $f_{i,\alpha}\restriction M_i=id_{M_i}$.
\end{fact}

Finally, to prove results about the uniqueness of limit models, we will additionally need to assume that non-$\mu$-splitting satisfies a symmetry property over limit models. We refine the definition of symmetry from \cite[Definition 3]{Va2} for non-$\mu$-splitting; this localization only requires symmetry to hold when $M_0$ is $(\mu, \delta)$-limit over $N$.

\begin{definition}\label{mu-delta symmetry}
Fix $\mu\geq\LS(\K)$ and $\delta$ a limit ordinal $<\mu^+$.
We say that an abstract elementary class exhibits \emph{$(\mu,\delta)$-symmetry for non-$\mu$-splitting} if  whenever models $M,M_0,N\in\K_\mu$ and elements $a$ and $b$  satisfy the conditions \ref{limit sym cond}-\ref{last} below, then there exists  $M^b$  a limit model over $M_0$, containing $b$, so that $\tp(a/M^b)$ does not $\mu$-split over $N$.  See Figure \ref{fig:sym}.
\begin{enumerate} 
\item\label{limit sym cond} $M$ is universal over $M_0$ and $M_0$ is a $(\mu,\delta)$-limit model over $N$.
\item\label{a cond}  $a\in M\backslash M_0$.
\item\label{a non-split} $\tp(a/M_0)$ is non-algebraic and does not $\mu$-split over $N$.
\item\label{last} $\tp(b/M)$ is non-algebraic and does not $\mu$-split over $M_0$. 
   
\end{enumerate}

\end{definition}

\begin{figure}[h]
\begin{tikzpicture}[rounded corners=5mm, scale=3,inner sep=.5mm]
\draw (0,1.25) rectangle (.75,.5);
\draw (.25,.75) node {$N$};
\draw (0,0) rectangle (3,1.25);
\draw (0,1.25) rectangle (1,0);
\draw (.85,.25) node {$M_0$};
\draw (3.2, .25) node {$M$};
\draw[color=gray] (0,1.25) rectangle (1.5, -.5);
\node at (1.1,-.25)[circle, fill, draw, label=45:$b$] {};
\node at (2,.75)[circle, fill, draw, label=45:$a$] {};
\draw[color=gray] (1.75,-.25) node {$M^{b}$};
\end{tikzpicture}
\caption{A diagram of the models and elements in the definition of $(\mu,\delta)$-symmetry. We assume the type $\tp(b/M)$ does not $\mu$-split over $M_0$ and $\tp(a/M_0)$ does not $\mu$-split over $N$.  Symmetry implies the existence of $M^b$ a limit model over $M_0$ so that $\tp(a/M^b)$  does not $\mu$-split over $N$.} \label{fig:sym}
\end{figure}

\begin{remark}\label{sym-rem}
In order to standardize notation, we will typically invoke `symmetry for $(M, M_0, N, a, b)$' to make the role of the models involved clear.

Note that $(\mu, \delta)$-symmetry is the same as $(\mu, \cf \delta)$-symmetry.  Also, the length of the limit model $M^b$ is not specified

\end{remark}

\section{Relatively Full Towers} \label{relfulltow-sec}

One approach to proving the uniqueness of limit models is to construct a continuous relatively full tower of length $\theta$, and then conclude that the union of the models in this tower is a $(\mu,\theta)$-limit model.  In this section we confirm that this approach can be carried out in this context, even if we remove continuity along the relatively full tower.

\begin{definition}[{\cite[Definition 3.2.1]{ShVi}}]\label{strong type
    defn}

For $M$ a $(\mu,\theta)$-limit model, \index{strong
    types}\index{Galois-type!strong}\index{$\St(M)$}\index{$(p,N)$}
    let
$$\St(M):=\left\{\begin{array}{ll}
(p,N)
& 
\left|\begin{array}{l}
N\prec_{\K}M;\\
N\text{ is a }(\mu,\theta)\text{-limit model};\\
M\text{ is universal over }N;\\
p\in \gaS(M)\text{ is non-algebraic}\\
\text{and }p\text{ does not }\mu\text{-split over }N.
\end{array}\right\}
\end{array}\right .
$$
Elements of $\St(M)$ are called {\em strong types.}
Two strong types $(p_1,N_1)\in\St(M_1)$ and $(p_2,N_2)\in\St(M_2)$
are
\emph{parallel} iff for every $M'$ of cardinality $\mu$ extending $M_1$
and $M_2$ there exists $q\in\gaS(M')$ such that $q$ extends both $p_1$
and $p_2$ and $q$ does not $\mu$-split over $N_1$ nor over $N_2$. 

\end{definition}

\begin{definition}[Relatively Full Towers]\label{def:relativefulltowers}
  Suppose that $I$ is a well-ordered set.  Let $(\bar M,\bar a,\bar N)$ be a tower indexed by $I$ such that
  each $M_i$ is a $(\mu,\sigma)$-limit model.  For each
$i$, let
$\langle M^\gamma_{i}\mid \gamma<\sigma\rangle$ witness that
$M_{i}$ is a
$(\mu,\sigma)$-limit model.\\
The tower
$(\bar M,\bar a,\bar N)$ is
\emph{full relative to
$(M^\gamma_{i})_{\gamma<\sigma,i\in I}$} iff 
\begin{enumerate} 
\item \label{niceorder-def} there exists a
  cofinal sequence $\langle i_\alpha\mid\alpha<\theta\rangle$ of $I$
  of order type $\theta$ such that
  there are $\mu\cdot \omega$ many elements between $i_\alpha$ and
  $i_{\alpha+1}$ and
\item\label{strong type condition} for every
$\gamma<\sigma$ and every $(p,M^\gamma_{i})\in\St(M_{i})$ with
$i_\alpha\leq i<i_{\alpha+1}$, there exists
$j\in I$   with $i\leq j< i_{\alpha+1}$ such that
$(\tp(a_j/M_j),N_j)$ and
$(p,M^\gamma_{i})$ are parallel.
\end{enumerate}
\end{definition}

The following proposition will allow us to use relatively full towers to produce limit models.  The fact that relatively full towers yield limit models was first proved in \cite{Va1} and in \cite{gvv} and later improved in \cite[Proposition 4.1.5]{Dr}.
We notice here that the proof of \cite[Proposition 4.1.5]{Dr} does not require that the tower be continuous and does not require that $\kappa^*_\mu(\K)=\omega$.  We provide the proof for completeness.

\begin{proposition}[Relatively full towers provide limit
  models]\label{relatively full is limit} Let $\theta$ be a limit ordinal
  $<\mu^+$ satisfying $\theta=\mu\cdot\theta$.  Suppose that $I$ is a
  well-ordered set
  as in Definition \ref{def:relativefulltowers}.(\ref{niceorder-def}).

Let $(\bar M,\bar a,\bar N)\in\K_{\mu,I}^*$ be a tower made up of
$(\mu,\sigma)$-limit models, for some fixed $\sigma$ with $\kappa^*_\mu(\K)\leq\cf(\sigma)<\mu^+$. If $(\bar M,\bar a,\bar
N)\in\K^*_{\mu,I}$ is full relative to $(M^\gamma_i)_{i\in  I,\gamma<\sigma}$, then
$M:=\Union_{i\in I}M_i$ is a $(\mu,\theta)$-limit model over $M_{i_0}$.
\end{proposition}

\begin{proof}
Because  the sequence $\langle i_\alpha\mid \alpha<\theta\rangle$ is cofinal in $I$ and  $\theta=\mu\cdot\theta$, we can rewrite $M:=\Union_{i\in I}M_i=\Union_{\beta<\theta}M_{i_{\beta}}=\Union_{\alpha<\theta}\Union_{\delta<\mu}M_{i_{\mu\alpha+\delta}}$.

For $\alpha<\theta$ and $\delta<\mu$, notice 
\begin{equation}\label{special equation}
M_{i_{\mu\alpha+\delta+1}}\text{ realizes every type over }M_{i_{\mu\alpha+\delta}}.
\end{equation}  
To see this take  $p\in\gaS(M_{i_{\mu\alpha+\delta}})$.  By our assumption that  $\cf(\sigma)\geq\kappa^*_\mu(\K)$, $p$ does not $\mu$-split over $M^\gamma_{i_{\mu\alpha+\delta}}$ for some $\gamma<\sigma$.  Therefore $(p,M^\gamma_{i_{\mu\alpha+\delta}})\in\St(M_{i_{\mu\alpha+\delta}})$.  By definition of relatively full towers, there is an $a_k$ with $i_{\mu\alpha+\delta}\leq k<i_{\mu\alpha+\delta+1}$ so that $(\tp(a_k/M_k),N_k)$ and $(p,M^\gamma_{i_{\mu\alpha+\delta}})$ are parallel.  Because $M_{i_{\mu\alpha+\delta}}\prec_{\K}M_k$, by the definition of parallel strong types, it must be the case that $a_k\models p$.

By a back and forth argument we can conclude from $(\ref{special equation})$ that $M_{i_{\mu\alpha+\mu}}$ is universal over $M_{i_{\mu\alpha}}$.  
Thus $M$ is a $(\mu,\theta)$-limit model.  

To see the details of the back-and-forth argument  mentioned in the previous paragraph, first translate $(\ref{special equation})$ to the terminology of \cite{Ba}: $(\ref{special equation})$ witnesses that  $\Union_{\beta<\mu}M_{i_{\mu\alpha+\beta}}$ is $1$-special over $M_{i_{\mu\alpha}}$.  Then, refer to the proof of Lemma 10.5 of \cite{Ba}.

\end{proof}

\section{Reduced Towers} \label{redtow-sec}

The proof of the uniqueness of limit models from \cite{sh394, gvv, Va1, Va1e} is two dimensional.  
In the context of towers, the relatively full towers are used to produce a $(\mu,\theta)$-limit model, but to conclude that this model is also a $(\mu,\omega)$-limit model, a $<$-increasing chain of $\omega$-many continuous towers of length $\theta+1$ is constructed.  We adapt this construction to prove Theorem \ref{uniqueness theorem}.  Instead of creating a chain of $\omega$-many towers, we produce a chain of $\delta$-many towers, and instead of each tower in this chain being continuous, we only require that these towers are continuous at limit ordinals of cofinality at least $\kappa^*_\mu(\K)$.

The use of towers should be compared with the proof of uniqueness of limit models in \cite[Section II.4]{shelahaecbook} (details are given in \cite[Section 9]{extendingframes}).  Both proofs create a `square' of models, but do so in a different way.  The proof here will proceed by starting with a 1-dimensional tower of models and then, in the induction step, extend this tower to fill out the square.  In contrast, the induction step of \cite[Lemma II.4.8]{shelahaecbook} adds single models at a time.  This seems like a minor distinction (or even just a difference in how the induction step is carried out), but there is a real distinction in the resulting squares.  In \cite{shelahaecbook}, the construction is `symmetric' in the sense that $\theta$ and $\delta$ are treated the same.  However, in the proof presented here, this symmetry is broken and one could `detect' which side of the square was laid out initially by observing where continuity fails.
  
      In  \cite{gvv, Va1, Va1e, Va2}, the continuity of the towers is achieved by restricting the construction to reduced towers, which under the stronger assumptions of  \cite{gvv, Va1, Va1e, Va2} are shown to be continuous.  We take this approach and notice that continuity of reduced towers at certain limit ordinals can be obtained with the weaker assumptions of Theorem \ref{uniqueness theorem}, in particular $\kappa^*_\mu(\K)<\mu^+$.

\begin{definition}\label{reduced defn}\index{reduced towers}
A tower $(\bar M,\bar a,\bar N)\in\K^*_{\mu,\alpha}$ is said to 
be \emph{reduced} provided that for every $(\bar M',\bar a,\bar
N)\in\K^*_{\mu,\alpha}$ with
$(\bar M,\bar a,\bar N)\leq(\bar M',\bar a,\bar
N)$ we have that for every
$i<\alpha$,
$$(*)_i\quad M'_i\cap\Union_{j<\alpha}M_j = M_i.$$
\end{definition}

The proofs of the following three results about reduced towers only require that the class $\K$ be stable in $\mu$ and that $\mu$-splitting satisfies the continuity property.  Although \cite{ShVi} works under stronger assumptions than we currently, none of these results use anything beyond Assumption \ref{assm}.  In particular, $\kappa^*_\mu(\K)=\omega$ holds in \cite{ShVi}, but is not used.

\begin{fact}[{\cite[Theorem 3.1.13]{ShVi}}]\label{density of
reduced}\index{reduced towers!density of} Let $\K$ satisfy Assumption \ref{assm}.  There exists a reduced
$<$-extension of every tower in
$\K^*_{\mu,\alpha}$.

\end{fact}

\begin{fact}[{\cite[Theorem 3.1.14]{ShVi}}]\label{union of reduced is reduced}
Let $\K$ satisfy Assumption \ref{assm}.  Suppose
$\langle (\bar M,\bar a,\bar N)^\gamma\in\K^*_{\mu,\alpha}\mid
\gamma<\beta\rangle$ is a $<$-increasing and continuous
sequence of
reduced towers such that the sequence is continuous in the sense that
for a limit $\gamma<\beta$, the tower $(\bar M,\bar a,\bar N)^\gamma$ is the union of the
towers $(\bar M,\bar a,\bar N)^\zeta$ for $\zeta<\gamma$.
Then the union of the sequence of towers $\langle (\bar M,\bar a,\bar N)^\gamma\in\K^*_{\mu,\alpha}\mid
\gamma<\beta\rangle$ is itself a
reduced tower.
\end{fact}

In fact the proof of Fact \ref{union of reduced is reduced} gives a slightly stronger result which allows us to take the union of an increasing chain of reduced towers of increasing index sets and conclude that the union is still reduced.

\begin{fact}[{\cite[Lemma 5.7]{gvv}}]\label{monotonicity}
Let $\K$ satisfy Assumption \ref{assm}. Suppose that $(\bar M,\bar a,\bar N)\in\K^*_{\mu,\alpha}$ is
reduced.  If $\beta<\alpha$, then $(\bar M,\bar
a,\bar N)\restriction \beta$ is reduced.
\end{fact}

The following theorem is related to \cite[Theorem 3]{Va2}, which additionally assumes that $\kappa^*_\mu(\K) = \omega$; in other words, it assumes that $\K$ is $\mu$-superstable.  Instead, we allow for strict stability (that is, $\kappa^*_\mu(\K)$ to be uncountable) at the cost of only guaranteeing continuity at limits of large cofinality.  In particular, the proof is similar to the proof of $(a)\to(b)$ in \cite[Theorem 3]{Va2}, but we crucially allow our towers to be discontinuous at $\gamma$ where $\cf(\gamma)<\kappa^*_\mu(\K)$.  At the good advice of the referee, we pull out the main construction--a tower amalgamation lemma, Lemma \ref{tow-ap-lem}--for future use.

\begin{theorem}\label{reduced are continuous}
Suppose $\K$  satisfies Assumption \ref{assm}.  
Let $\alpha$ be an ordinal and $\delta$ be a regular ordinal satisfying
$$\kappa^*_\mu(\K) \leq \delta < \alpha < \mu^+$$
If $\K$ satisfies $(\mu, \delta)$-symmetry for non-$\mu$-splitting, then all reduced towers in $\K^*_{\mu, \alpha}$ are continuous at $\delta$ (i.e., $M_\delta=\Union_{\beta<\delta}M_\beta$).
\end{theorem}
The proof is below, but first we prove the crucial lemma.  In it we are given a long base tower $\T^0$ and a shorter tower $\T^1$ that extends the initial segment of the base tower.  Crucially, we have an element $b$ that appears in $\T^1$ but not $\T^0$ and we want to lengthen and extend $\T^1$ to $\T^2$, which extends $\T^0$ (and contains $b$).  To do this, we have the additional assumption (Lemma \ref{tow-ap-lem}.(2)) that there are nice intermediate models $\bar{M}^+$ interleaved in the base tower.  This might seem artificial in the abstract, but is precisely given to us in the proof of the main theorem.

Throughout this lemma and following proof, we use the notation that if a tower is indicated by a superscript, then we use the same superscript to pick out the component models, e.g., when we write $\T^0$ for a tower, the $M^0_j$'s are used to denote the models in the corresponding $\bar{M}$ sequence.  Because a key idea is the continuity of towers, it will be useful to use the following notation for any sequence 
$\{M^x_i : i < \alpha\}$ (with `$x$' standing in for some superscript):
$$M^x_{<j} := \bigcup_{i < j} M^x_i$$
If $M^x_{<j} = M^x_j$ (for limit $j$), then the sequence is continuous at $j$.

\begin{lemma}[Tower amalgamation lemma]\label{tow-ap-lem}
Suppose $\K$ satisfies Assumption \ref{assm} and $(\mu, \beta)$-symmetry.  If $\alpha < \beta$ and we have towers $\T^0 \in \K^*_{\mu, \beta}$ and $\T^1 \in \K^*_{\mu, \alpha}$ such that
\begin{enumerate}
	\item $\T^0 \rest \alpha \leq \T^1$;
	\item for all $j < \beta$, there is a model $M^{+j}$ such that
	\begin{enumerate}
		\item $M^0_j \preck M^{+j} \preck M^0_{j+1}$;
		\item $M^{+j}$ is $(\mu, \beta)$-limit over $N^0_j$; and
		\item $\tp(a^0_j/M^{+j})$ does not $\mu$-split over $N^0_j$;
	\end{enumerate}
	\item there is $\gamma < \alpha$ and $b \in M_{<\alpha}^1$ such that
	\begin{enumerate}
		\item $b \not \in M^0_{<\beta}$; and
		\item $\tp(b/M^0_{<\beta})$ does not $\mu$-split over $M^0_\gamma$;
	\end{enumerate}
\end{enumerate}
then there is $\T^2 \in \K^*_{\mu, \beta}$ such that
\begin{enumerate}
	\item $\T^0 < \T^2$;
	\item there is $f:\T^1 \to \T^2 \rest \alpha$ such that $f$ is the identity on $M^0_{<\beta}$; and
	\item $b \in M^2_{<\beta}$.
\end{enumerate}
\end{lemma}

\begin{proof}
We are going to build a directed sequence of approximations to $\T^2$.  That is, we will build, for $\alpha \leq j \leq k \leq \beta$,
$$\T^{2, k} \in \K^*_{\mu, j} \text{ and } f_{j, k}:\T^{2,j} \to \T^{2, k}$$
such that, for all such $j \leq k$ from $[\alpha, \beta]$, we have
\begin{enumerate}
	\item \label{a} $\T^{2, \alpha} = \T^1$;
	\item \label{b} $\T^0\rest j \leq \T^{2, j}$;
	\item \label{c} $f_{j, k}\left(\T^{2, j}\right) \leq \T^{2, k} \rest j$;
	\item \label{d} $f_{j, k} \rest M^0_j$ is the identity;
	\item \label{e} $M^{2, j+1}_{j+1}$ is universal over $f_{j, j+1}\left(M^{2, j}_j\right)$;
	\item \label{f} $b \in M^{2, j}_{<\alpha}$; and
	\item \label{g} $\tp(f_{j, k}(b)/M^0_k)$ does not $\mu$-split over $M^0_\gamma$.
\end{enumerate}
This is enough by taking $\T^2 = \T^{2,\beta}$ and $f = f_{\alpha, \beta}$.  We now describe the construction based on induction on $k \in [\alpha, \beta]$.
\begin{enumerate}
	\item \underline{{\bf Case 1:}} $k = \alpha$\\
	
	Take $\T^{2, \alpha} = \T^1$ and $f_{\alpha, \alpha}$ to be the identity.\\
	
	\item \underline{{\bf Case 2:}} $k > \alpha$ is limit\\
	
	By induction, we have the directed system below $k$.  Note that, for all $j < k$,
	\begin{enumerate}
		\item $\T^0 \rest j \leq \T^{2, j}$ by (\ref{b});
		\item $f_{i, j} \rest M^0_i$ is the identity for $i < j$ by (\ref{d}); and
		\item $M^{2, j+1}_{j+1}$ is universal over $f_{j, j+1}(M^{2,j}_j)$.
	\end{enumerate}
	By Fact \ref{limit stage prop}, we can find a direct limit $\hat{\T}^k \in \K^*_{\mu, k}$ and $\hat{f}_{j, k}: \T^{2, j} \to \hat{\T}^k \rest j$ for all $j<k$ with the additional property that $\hat{f}_{j, k} \rest M^0_j$ is the identity and $\T^0\rest k \leq \hat{\T}^k$; note we decorate these symbols since they are \emph{not} the desired objects since we have to account for $b$.
	
	First, note that $\hat{f}_{j, k}(M^{2, j}_j)$ is universal over $M^0_j$ for all $j < k$ becasue $\hat{\T}^k$ extends $\T^0$ and this extension at the coordinate $j$ is proper.\\
	
	{\bf Claim:} The model $\hat{M}^k_{<k}$ is limit.
	\begin{proof}
	Condition (\ref{e}) gives us that $f_{j, j+1}:M^{2,j}_j \to M^{2, j+1}_{j+1}$ is a universal embedding.  So the direct limit is $(\mu, \cf k)$-limit and that limit is the desired model.
	\end{proof}
	
	{\bf Claim:} For $j \in [\alpha, k)$, $\tp\left(\hat{f}_{\alpha, k}(b)/M^0_j\right)$ does not $\mu$-split over $M^0_\gamma$.
	\begin{proof}
	Note that $\hat{f}_{\alpha, k} = \hat{f}_{j, k} \circ f_{\alpha, j}$.  By (\ref{g}), we know that $\tp\left(f_{\alpha, j}(b)/M^0_j\right)$ does not $\mu$-split over $M^0_\gamma$.  When we apply $\hat{f}_{j,k}$ to this and note that, by (\ref{d}) and $M^0_\gamma \preck M^0_j$, we get the desired result.
	\end{proof}
	Since  $M^0_{j+1}$ is universal over $M^0_j$, we can apply the continuity of non-$\mu$-splitting (Assumption \ref{assm}.(4)) to get
	$$\tp\left(\hat{f}_{\alpha, k}(b)/M^0_{<k} \right) \text{ does not $\mu$-split over }M^0_\gamma$$
	Whether this domain is $M^0_k$ depends on the continuity of the tower $\T^0$ at $k$. Since $\hat{f}_{\alpha, k}$ fixes $M^0_{\gamma+1}$, we have that
$$\tp(b/M^0_{\gamma+1}) = \tp(\hat{f}_{\alpha, k}(b) / M^0_{\gamma+1})$$
	By monotonicity of non-splitting and the Hypothesis 1.3.(b), we know that $\tp(b/M^0_{<k})$ does not $\mu$-split over $M^0_\gamma$.  By uniquessness of non-splitting extensions (see \cite[Theorem 12.7 and Exercise 12.8]{Ba}), we have that
	$$\tp(b/ M^0_{<k}) = \tp(\hat{f}_{\alpha, k}(b)/M^0_{<k})$$
	This gives an automorphism $g$ that fixes $M^0_{<k}$ and sends $\hat{f}_{\alpha, k}(b)$ to $b$.  We now define the desired objects 
	$$\T^{2, k} = g(\hat{\T}^k) \text{ and } f_{j, k} = g \circ \hat{f}_{j, k}$$
	
	We know want to show this construction works.
	\begin{enumerate}
		\item[\ref{b}.] $\hat{\T}^k \geq \T^0\rest k$ by construction and $g$ fixes $M^0_{<k}$, which are the models in $\T^0\rest k$.  Thus, $\T^0\rest k \leq g(\hat{\T}^k) = \T^{2, k}$.
		\item[\ref{c}.] By construction,
		$$\hat{f}_{j, k}(\T^{2, j}) \leq \hat{\T}^k \rest j$$
		Thus,
		$$f_{j, k}(\T^{2, j}) = g \circ \hat{f}_{j, k}(\T^{2, j}) \leq g(\hat{\T}^k \rest j) = \T^k\rest j$$
		\item[\ref{d}.] $g$ and $\hat{f}_{j,k}$ are both the identity on $M^0_j$, so $g\circ \hat{f}_{j,k}$ is too.
		\item[\ref{f}.] $b \in M_\gamma^{2, \alpha} \preck M_\alpha^{2, \alpha}$, so $g\left(\hat{f}_{\alpha, k}(b)\right) = b \in M_\gamma^{2, k}$.\\
	\end{enumerate}
		
	\item \underline{{\bf Case 3:}} $i = k+1$ where $k$ is limit\footnotei{WB: We removed the extraneous hypothesis about discontinuity}\\
	
	We have $\T^{2, k}$ and $f_{j, k}$ defined.  This is where we use the $M^{+j}$ models from clause 2 of the hypothesis.  By hypothesis 3.(b), we have that 
	$$\tp\left(b/M^0_{<\beta}\right) \text{ does not $\mu$-split over } M^0_\gamma$$
	Now we will apply $(\mu, \beta)$-symmetry to $(M^0_{<\beta}, M^{+K}, N^k, a^0_k, b)$ (recall Remark \ref{sym-rem}); see Figure \ref{fig:successor}.  Note that our hypothesis includes $M^0_{<\beta}$ is universal over $M^{+k}$ and $M^{+k}$ is $(\mu, \beta)$-limit over $N_k$; this is crucially the instances of symmetry we are assuming.  The symmetry gives us $M^b_k \in \K_\mu$ that is limit over $M^{+k}$ such that $\tp(a^0_k/M^b_k)$ does not $\mu$-split over $N_k$.  This model $M^b_k$ lives outside of the towers we have built so far, but we can find $M^* \in \K_\mu$ containing $M^{2, k}_{<k}$ and $M^{+k}$.  
	
	\begin{figure}[h]
\begin{tikzpicture}[rounded corners=5mm, scale=3,inner sep=.5mm]
\draw (0,1.25) rectangle (.75,.5);
\draw (.25,.75) node {$N_{k}$};
\draw (0,0) rectangle (3,1.25);
\draw (0,1.25) rectangle (1,0);
\draw (.8,.25) node {$M^{+k}$};
\draw (3.35, .25) node {$M^0_{<\beta}$};
\draw[color=gray] (0,1.25) rectangle (1.5, -.5);
\node at (1.1,-.25)[circle, fill, draw, label=45:$b$] {};
\node at (2,.75)[circle, fill, draw, label=45:$a_{k}$] {};
\draw[color=gray] (1.75,-.25) node {$M^{b}_{k}$};
\end{tikzpicture}
\caption{A diagram of the application of $(\mu,\delta)$-symmetry in the successor stage of the directed system construction in the proof of  Theorem \ref{reduced are continuous}. We have $\tp(b/ M^0_{<\beta})$ does not $\mu$-split over $M^{+k}$ and $\tp(a_{k}/M^{+k})$ does not $\mu$-split over $N_{k}$.  Symmetry implies the existence of  $M^b_k$ a limit model over $M^{+k}$. so that $\tp(a_{k}/M^b)$  does not $\mu$-split over $N_{k}$.} \label{fig:successor}
\end{figure}
	
	We know that
	$$M^0_k \preck {}_{, limit} M^{+k} \preck {}_{, limit} M^b_k$$
	so we can find $f:M^*\xrightarrow[M^0_k]{} M^{+k}$, which gives 
	$$f\left(M^{2,j}_{<k}\right) \preck{}_{, univ} M^b_k$$
	We have that $\tp(b/M^0_k)$ does not $\mu$-split over $M^0_\gamma$ and $f$ fixes $M^0_k$.  By the uniqueness of non-splitting extension and $M^0_k \preck{}_{univ} M^b_k$, we can find an automorphism $g$ fixing $M^0_k$ such that $g(f(b))=b$ and
	$$g\left(f(M^*)\right) \preck M^b_k$$
	Now, set $M^{2, k+1}_k$ to be an extension of $M^b_k$ that is universal over $M^0_k$ and set $f_{k, k+1} = g \circ f$; see Figure \ref{fig:reduced tower}.
	
	\begin{figure}[h]
\begin{tikzpicture}[scale=2.9,inner sep=.5mm]
\tikzstyle{rrect}=[rounded corners=5mm]
\draw (0,0) [rrect] rectangle (3.5,1);
\draw (0,1)[rrect] rectangle (1.1,0);
\draw (0,1)[rrect] rectangle (2.1,0);
\draw (0,1)[rrect] rectangle (2.55,0);
\draw (0,1)[rrect] rectangle (2.7,0);
\draw (.85,.25) node {$M^{0}_{\gamma+1}$};
\draw (1.75,.25) node {$M^0_{<k}$};
\draw (2.35,.25) node {$M^0_{k}$};
\draw (3,.25) node {$M^{+k}\dots$};
\draw (3.75,.25) node {$M^0_{<\beta}$};
\draw (-.25,.25) node {$\T^{0}$};
\node at (.5,-.25)[circle, fill, draw, label=270:$b$] {};
\draw[line width=.75mm] (0,1) [rrect] rectangle (2.1, -.45);
\draw[line width=.75mm] (0,1) [rrect] rectangle (1.1, -.45);
\draw (0,1) [rrect] rectangle (2.55, -.45);
\draw (2.4,-.6) node {$M*$};
\draw [->, shorten >=3pt] (2.5,-.3) to [bend right=65] node[pos=0.3,below] {$f$}(2.6,.1);
\draw [->, shorten >=3pt] (0,-.3) to [bend right=95] node[pos=0.6,below] {$g\circ f$}(1.5,-.75);
\draw (.85,-.2) node {$M^{2,k}_{\gamma+1}$};
\draw (1.75,-.2) node {$\dots M^{2, k}_{<k}$};
\draw (3.4,-.2) node {$M^b_k$};
\draw (-.25,-.2) node {$\T^{2,k}$};
\draw (-.25,-.75) node {$\T^{2,k+1}$};
\node at (2.95,.75)[circle, fill, draw, label=315:$a_{k}$] {};
\node at (1.2,.75)[circle, fill, draw, label=315:$a_{\gamma+1}$] {};
\begin{scope}
  \clip (0,1) [rrect] rectangle (5,-1);
\draw (0,1.8) [rrect, xslant=-0.4, dotted] rectangle (3.2, -.75); 
\end{scope}

\end{tikzpicture}
\caption{The construction of $\T^{2, k+1}$(dotted) from $\T^{2,k}$ (bold) with $g\circ f$ fixing $M^0_k$ and $b$.} \label{fig:reduced tower}
\end{figure}
	
	We can define the rest of the tower by, for $j<k$, 
	\begin{eqnarray*}
	M^{2, k+1}_j &:=& f_{k, k+1}(M^{2,k}_j)\\
	f_{j, k+1} &:=& f_{k, k+1}\circ f_{j, k}
	\end{eqnarray*}
	\item The other successor cases are unchanged, i.e., as in \cite[Theorem 3]{Va2}.
\end{enumerate}

\end{proof}

\begin{proof}[Proof of Theorem \ref{reduced are continuous}]
We work by contradiction; to that end, fix some $\delta < \alpha$ in $[\kappa^*_\mu(\K), \mu^+)$ such that $(\mu, \delta)$-symmetry holds for non-$\mu$-splitting and there is a reduced tower that is not continuous at $\delta$.  By Fact \ref{monotonicity}, this counter-example remains reduced if we restrict it to length $\delta+1$.  Thus, there is a reduced
$$\T = (\bar M,\bar a,\bar N) \in \K^*_{\mu, \delta+1}$$
that is discontinuous at $\delta$.  This discontinuity is witnessed by some element $b \in M_\delta - M_{<\delta}$.

We will use this discontinuity to contradict that $\T$ is reduced by finding $\mathring{\T} \in \K^*_{\mu, \delta}$ extending $\T^{diag}$ which extends $\T\rest \delta$ such that $\mathring{\T}$ contains $b$.

First, we build a series of extensions of $\T\rest \delta$.  Build $\T^i = (\bar{M}^i, \bar{a}^i, \bar{N}^i) \in \K^*_{\mu, \delta}$ for $i \leq \delta$ as a $<$-increasing chain of reduced towers as follows:
\begin{enumerate}
	\item $\T^0 = \T\rest \delta$.
	\item At limits, take unions; by Fact \ref{union of reduced is reduced}, this results in a reduced tower.
	\item Given $\T^i$, apply Fact \ref{density of reduced} $\delta$-many times to find reduced $\T^{i+1} > \T^i$ so that $M^{i+1}_\beta$ is $(\mu, \delta)$-limit over $M^i_\beta$ for all $\beta < \delta$.  Note that this implies that $M^{i+1}_\beta$ is $(\mu, \delta)$-limit over $N_\beta$.
\end{enumerate}

Set
$$M^* = \Union_{\beta < \delta} M^\delta_\beta = \Union_{i<\delta}  M^i_i$$
Figure \ref{fig:Mstar} is an illustration of these models.  $M^*$ contains all elements in any of the $\T^i$'s.  Note that $M^*$ is $(\mu, \delta)$-limit since $M^{i+1}_{i+1}$ is universal over $M_i^i$.\\

\begin{figure}[h]
\begin{tikzpicture}[rounded corners=5mm,scale =2.9,inner sep=.5mm]
\draw (0,1.5) rectangle (.75,.5);
\draw (0,1.5) rectangle (1.75,1);
\draw (.25,.75) node {$N_0$};
\draw (1.25,1.25) node {$N_\beta$};
\draw (0,0) rectangle (4,1.5);
\draw (0,1.5) rectangle (3.5,-2);
\draw (.85,.25) node {$M_0$};
\draw(1.25,.25) node {$M_1$};
\draw (1.75,.25) node {$\dots M_\beta$};
\draw (2.35,.25) node {$M_{\beta+1}$};
\draw (3.15,.2) node {$\dots\displaystyle{M_{<\delta}}$};
\draw (3.85, .25) node {$M_\delta$};
\draw (-.5,.25) node {$(\bar M,\bar a,\bar N)$};
\draw (0,1.5) rectangle (3.5, -.4);
\draw (.85,-.15) node {$M^{1}_0$};
\draw (1.75,-.15) node {$\dots M^{1}_\beta$};
\draw (2.35,-.15) node {$M^{1}_{\beta+1}$};
\draw(1.25,-.15) node {$M^{1}_1$};
\draw (3.15,-.2) node {$\dots\displaystyle{M^{1}_{<\delta}}$};
\draw (-.5,-.15) node {$(\bar M,\bar a,\bar N)^1$};
\draw (.85,-.6) node {$\vdots$};
\draw (1.75,-.6) node {$\vdots$};
\draw (2.35,-.6) node {$\vdots$};
\draw (3.2,-.6) node {$\vdots$};
\draw (0,1.5) rectangle (3.5, -1);
\draw (.85,-.85) node {$M^{i}_0$};
\draw (1.75,-.85) node {$\tiny{\dots} M^{i}_\beta$};
\draw (2.35,-.85) node {$M^{i}_{\beta+1}$};
\draw (3,-.85) node {$\dots M^{i}_{<\delta}$};
\draw (-.5,-.85) node {$(\bar M,\bar a,\bar N)^i$};
\draw (0,1.5) rectangle (3.5, -1.35);
\draw (0,1.5) rectangle (1,-2);
\draw(0,1.5) rectangle (1.5, -2);
\draw (0,1.5) rectangle (2.5, -2);
\draw (0,1.5) rectangle (2,-2);
\draw (0,1.5) rectangle (3.5, -2);
\draw (.8,-1.15) node {$M^{i+1}_0$};
\draw (1.8,-1.15) node {$ M^{i+1}_\beta$};
\draw (2.3,-1.15) node {$M^{i+1}_{\beta+1}$};
\draw (3,-1.2) node {$\dots M^{i+1}_{<\delta}$};
\draw (-.5,-1.15) node {$(\bar M,\bar a,\bar N)^{i+1}$};
\draw (.85,-1.6) node {$\vdots$};
\draw (1.75,-1.6) node {$\vdots$};
\draw (2.35,-1.6) node {$\vdots$};
\draw (3.2,-1.6) node {$\vdots$};
\node at (3.75,.75)[circle, fill, draw, label=90:$b$] {};
\node at (2.25,.75)[circle, fill, draw, label=290:$a_\beta$] {};
\node at (1.1,.75)[circle, fill, draw, label=290:$a_1$] {};
\draw (3.75,-1.75) node {$M^*$};
\end{tikzpicture}
\caption{$(\bar M,\bar a,\bar N)$ and the  towers $(\bar M,\bar a,\bar N)^i$ extending $(\bar M,\bar a,\bar N)\restriction\delta$ that don't contain $b$.} \label{fig:Mstar}
\end{figure}

{\bf Claim:} $b \not\in M^*$

\begin{proof}
For contradiction, assume that $b \in M^*$.  This means that there is some $\beta < \delta$ such that $b \in M^\delta_\beta$.  Note that $\T^\delta > \T\rest \delta$.  So we can find $M^\delta_\delta \succ M^*$ that is also a universal extension of $M_\delta$.  Then we have
$$\T \leq \T^\delta{}^\frown\langle M^\delta_\delta \rangle$$
Since $\T$ is reduced by assumption, we have
$$M^\delta_\beta \cap M_\delta = M_\beta$$
However, this is a contradiction because the left-hand side contains $b$ but the right-hand side does not.
\end{proof}

Thus, $\tp(b/M^*)$ is non-algebraic.  We are interested in the diagonal sequence $M^i_i$ for $i <\delta$, so set
$$\T^{diag} = (M_i^i, a_i, N_i)_{i < \delta} \in \K^*_{\mu, \delta}$$\\

{\bf Claim:} There is $i^* < \delta$ so $\tp(b/M^*)$ does not $\mu$-split over $M^{i^*}_{i^*}$.\\
\begin{proof}
First, we smooth the diagonal sequence to be continuous by setting 
$$\hat{M}_i  = \begin{cases} M^i_i & \text{ if $i$ is successor}\\  M^i_{<i} & \text{ if $i$ is limit}\end{cases}$$
Since $M^*$ is $(\mu, \delta)$-limit and $\cf \delta \geq \kappa^*_\mu(\K)$ (and by the monotonicity of non-splitting), there is a successor $i^* < \delta$ so
$$\tp(b/M^*) \text{ does not $\mu$-split over }\hat{M}_{i^*} = M^{i^*}_{i^*}$$
\end{proof}

Note that $\T\rest \delta \leq \T^{diag}$.  We are going to build a series of towers towards finding a contradiction.

First, we build $\T^b \in \K^*_{\mu, i^*+2}$ containing $b$ in the final model and that extends $\T^{diag} \rest (i^*+2)$.  To build this, note that $M^{i^*}_{i^*}$ is $(\mu, \delta)$-limit over $N_{i^*}$.

Now we invoke $(\mu, \delta)$-symmetry for $(M^*, M^{i^*}_{i^*}, N_{i^*}, a_{i^*}, b)$ (recall Remark \ref{sym-rem}).  This gives $M_b \in \K_\mu$ which is limit over $M^{i^*}_{i^*}$, contains $b$, and for which
$$\tp(a_{i^*}/M^b)\text{ does not $\mu$-split over }N_{i^*}$$
Now we can define 
$$\T^b := \T^{diag} \rest (i^*+1) {}^\frown \langle M^b \rangle$$
We now wish to apply Lemma \ref{tow-ap-lem}.  To establish the hypothesis of that lemma, note
\begin{enumerate}
    \item $\T^{diag} \rest (i^*+2) \leq \T^b$; and
    \item[3.] $b \in M^b_{<i^*+2} - M^{diag}_{<\delta}$
\end{enumerate}
Additionally, we can take $M^{j+1}_j$ for what is called $M^{+j}$ there as
\begin{enumerate}
    \item[(a)] $M^{diag}_j \preck M^{j+1}_j \preck M^{diag}_{j+1}$;
    \item[(b)] $M^{j+1}_j$ is $(\mu, \delta)$-limit over $N_j$; and
    \item[(c)] $\tp(a^{diag}_j/M^{j+1}_j)$ does not $\mu$-split over $N_j$.
\end{enumerate}

Thus, Lemma \ref{tow-ap-lem} gives us a tower $\mathring{\T} \in \K^*_{\mu, \delta}$ such that
\begin{enumerate}
    \item $\T^{diag} < \mathring{\T}$; and
    \item $b \in \mathring{M}_{<\delta}$ (in particular, there is $\gamma < \delta$ such that $b \in \mathring{M}_\gamma$).
\end{enumerate}
We do not need the other clause of the lemma.  Then we have that $\T\rest \delta \leq \T^{diag} < \mathring{\T}$.  We can extend this by finding $\mathring{M}_\delta$ to be an extension of $\mathring{M}_{<\delta}$ that is universal over $M_\delta$.  Then we have
$$\T \leq \mathring{\T}^\frown\langle\mathring{M}_\delta\rangle$$
Since $\T$ is reduced by assumption, we must have that
$$\mathring{M}_\gamma \cap M_\delta = M_\gamma$$
However, this is a contradiction since $b$ is in the left-hand side, but not the right-hand side.
\end{proof}

Although not used here, the converse of this theorem is also true, as in \cite{Va2}.  Note that the following does not have any assumption about $\kappa^*_\mu(\K)$.

\begin{proposition}
Suppose $\K$  satisfies Assumption \ref{assm}.(1), (2), and (4).  
Suppose further that that, for every reduced tower $(\bar M, \bar a, \bar M) \in \K_{\mu, \alpha}^*$, $\bar M$ is continuous at limit ordinals of cofinality $\delta$.  Then $\K$ satisfies $(\mu, \delta)$-symmetry for non $\mu$-splitting.
\end{proposition}

\begin{proof}
The proof is an easy adaptation of \cite[Theorem 3.$(b)\to(a)$]{Va2}.  The same argument works; the only adaptations are to require that every limit model to in fact be a $(\mu, \delta)$ limit model and that the tower $\T$ be of length $\delta+1$\footnote{In a happy coincidence, the notation in that proof already agrees with this change.}.
\end{proof}

\section{Uniqueness of Long Limit Models} \label{ulm-sec}
We now begin the proof Theorem \ref{uniqueness theorem}, which we restate here.
\begin{theorem1}
Suppose that $\K$ is an abstract elementary class satisfying Assumption \ref{assm}.
 For $\theta$ and $\delta$ limit ordinals $<\mu^+$ both with cofinality $\geq\kappa^*_\mu(\K)$,
if $\K$ satisfies symmetry for non-$\mu$-splitting (or just $(\mu,\delta)$-symmetry), then, for any $M_1$ and $M_2$ that are $(\mu,\theta)$ and $(\mu,\delta)$-limit models over $M_0$, respectively, we have that $M_1$ and $M_2$ are isomorphic over $M_0$.

\end{theorem1}
The structure of the proof of Theorem \ref{uniqueness theorem} from this point on is similar to the proof in \cite[Theorem 1.9]{gvv}.  For completeness we include the details here, and emphasize the points of departure from \cite[Theorem 1.9]{gvv}.

We construct an array of models which will produce a model that is both a $(\mu,\theta)$- and a $(\mu,\delta)$-limit model.
Let $\theta$ be an ordinal as in the definition of relatively full tower so that $\cf(\theta)\geq\kappa^*_\mu(\K)$ and let $\delta=\kappa^*_\mu(\K)$.  The
goal is to build an array of models with $\delta+1$ rows so that the
bottom row of the array is a relatively full tower indexed by a set of cofinality $\theta+1$ continuous at $\theta$. 
 To do this, we will be
adding elements to the index set of towers row by row so that 
at stage $n$ of our construction the
tower that we build is indexed by $I_n$
 described here.

The index sets $I_\beta$ will be defined inductively so that $\langle I_\beta\mid
\beta<\delta+1\rangle$ is an increasing and continuous chain of well-ordered sets. We
fix $I_0$ to be an index set of order type $\theta+1$ and will denote it by
$\langle i_\alpha\mid\alpha\leq\theta\rangle$.  We will refer to the members of $I_0$ by name in many
stages of the construction.  These indices serve as anchors for the
members of the remaining index sets in the array.  Next we demand that
for each $\beta<\delta$, $\{j\in I_\beta\mid i_\alpha<j<i_{\alpha+1}\}$ has order type $\mu\cdot \beta$ such
that each $I_\beta$ has supremum $i_\theta$. An example of such $\langle I_\beta\mid \beta\leq\delta\rangle$
is $I_\beta=\theta\times(\mu\cdot \beta)\Union\{i_\theta\}$ ordered lexicographically, where $i_\theta$ is
an element $\geq$  each $i\in \Union_{\beta<\delta}I_\beta$. Also, let $I=\bigcup_{\beta<\delta}I_\beta$.

To prove Theorem \ref{uniqueness theorem}, we need to prove that, for a
fixed $M\in\K$ of cardinality $\mu$, any $(\mu,\theta)$-limit and $(\mu,\delta)$-limit model
over $M$ are isomorphic over $M$.  Since all $(\mu, \theta)$-limits over $M$ are isomorphic over $M$ (and the same holds for $(\mu, \delta)$-limits), it is enough to construct a single model that is simultaneously $(\mu, \theta)$-limit and $(\mu, \delta)$-limit over $M$.  Let  us begin by fixing a limit model $M\in\K_\mu$.
We define, by induction on $\beta\leq\delta$, a $<$-increasing and continuous
sequence of towers $(\bar M,\bar a,\bar N)^\beta$ such that
\begin{enumerate}
\item $\T^0:=(\bar M,\bar a,\bar N)^0$ is a tower with $M^0_0=M$.
\item $\T^\beta:=(\bar M,\bar a,\bar N)^\beta\in\K^*_{\mu,I_\beta}$.
\item\label{realizing types} For every $(p,N)\in\St(M^\beta_i)$ with $i_\alpha\leq i< i_{\alpha+1}$ there is 
$j\in I_{\beta+1}$ with $i_\alpha< j<i_{\alpha+1}$ so that $(\tp(a_j/M^{\beta+1}_j),N^{\beta+1}_j)$ and $(p,N)$ 
are parallel.  

\end{enumerate}
See Figure \ref{fig:arrayconstruction}.

\begin{figure}[h]
\begin{tikzpicture}[rounded corners=5mm,scale =2.9,inner sep=.5mm]
\draw (0,1.5) rectangle (.75,.5);
\draw (0,1.5) rectangle (1.75,1);
\draw (.25,.75) node {$N_{i_0}$};
\draw (1.25,1.25) node {$N_{i_\alpha}$};
\draw (0,0) rectangle (3.7,1.5);
\draw (0,1.5) rectangle (3.7,-2);
\draw (.85,.25) node {$M^0_{i_0}$};
\draw(1.25,.25) node {$M^0_{i_1}$};
\draw (1.75,.25) node {$\dots M^0_{i_\alpha}$};
\draw (2.35,.25) node {$M^0_{i_{\alpha+1}}$};
\draw (3.15,.2) node {$\dots\substack{M^0_{i_\theta}=\\ \Union_{k<\theta}M^0_{i_k}}$};
\draw (-.5,.25) node {$\T^0\in\K^*_{\mu,I_0}$};
\draw (0,1.5) rectangle (3.7, -.5);
\draw (.85,-.15) node {$M^{1}_{i_0}$};
\draw (1.75,-.15) node {$\dots M^{1}_{i_\alpha}$};
\draw (2.3,-.15) node {$\lll M^{1}_{i_{\alpha+1}}$};
\draw(1.25,-.15) node {$\lll M^{1}_{i_1}$};
\draw (3.15,-.2) node {$\dots\substack{ \runiv \\ M^1_{i_\theta}=\\ \Union_{k<\theta}M^{1}_{i_k}}$};
\draw (-.5,-.15) node {$\T^1\in\K^*_{\mu,I_1}$};
\draw (.85,-.6) node {$\vdots$};
\draw (1.75,-.6) node {$\vdots$};
\draw (2.35,-.6) node {$\vdots$};
\draw (3.2,-.6) node {$\vdots$};
\draw (0,1.5) rectangle (3.7, -1);
\draw (.85,-.85) node {$M^{\beta}_{i_0}$};
\draw (1.75,-.85) node {$\tiny{\dots} M^{\beta}_{i_\alpha}$};
\draw (2.35,-.85) node {$M^{\beta}_{i_{\alpha+1}}$};
\draw (3,-.85) node {$\dots \substack{M^\beta_{i_\theta}=\\\Union_{k<\theta}M^{\beta}_{i_k}}$};
\draw (-.5,-.85) node {$\T^\beta\in\K^*_{\mu,I_\beta}$};
\draw (0,1.5) rectangle (3.7, -1.45);
\draw (0,1.5) rectangle (1,-2);
\draw(0,1.5) rectangle (1.5, -2);
\draw (0,1.5) rectangle (2.6, -2);
\draw (0,1.5) rectangle (2,-2);
\draw (0,1.5) rectangle (3.7, -2);
\draw (.8,-1.15) node {$M^{\beta+1}_{i_0}$};
\draw (1.8,-1.15) node {$ M^{\beta+1}_{i_\alpha}$};
\draw (2.3,-1.15) node {$\lll M^{\beta+1}_{i_{\alpha+1}}$};
\draw (3,-1.2) node {$\dots\substack{\runiv\\ M^{\beta+1}_{i_\theta}= \\\Union_{k<\theta}M^{\beta+1}_{i_k}}$};
\draw (-.5,-1.15) node {$\T^{\beta+1}\in\K^*_{\mu,I_{\beta+1}}$};
\draw (-.5, -1.8) node {$\T^{\delta}\in\K^*_{\mu,I_{\delta}}$};
\draw (.85,-1.55) node {$\vdots$};
\draw (1.75,-1.55) node {$\vdots$};
\draw (2.35,-1.55) node {$\vdots$};
\draw (3.2,-1.55) node {$\vdots$};
\draw (.8,-1.8) node {$M^{\delta}_{i_0}$};
\draw (1.8,-1.8) node {$ M^{\delta}_{i_\alpha}$};
\draw (2.3,-1.8) node {$\prec_{u} M^{\delta}_{i_{\alpha+1}}$};
\node at (2.1,.75)[circle, fill, draw, label=290:$a_{i_\alpha}$] {};
\node at (1.1,.75)[circle, fill, draw, label=290:$a_{i_1}$] {};
\draw (3.1,-1.8) node {$M^\delta_{i_\theta}=\displaystyle{\Union_{\gamma<\delta, k<\theta}M^\gamma_{i_k}}$};
\end{tikzpicture}
\caption{The chain of length $\delta$ of towers of increasing index sets $I_j$ of cofinality $\theta+1$.  The symbol $\lll$ indicates that there are $\mu$ many new indices between $i_\beta$ and $i_{\beta+1}$ in $I_{j+1}\backslash I_j$.  The elements indexed by these indices realize all the strong types over the model $M^j_{i_\alpha}$.  The notation $\prec_{u}$ is an abbreviation for a universal extension.} \label{fig:arrayconstruction}
\end{figure}

Given $M$, we can find a tower $(\bar M,\bar a,\bar N)^0\in\K^*_{\mu,I_0}$
with $M\preceq_{\K}M^0_0$ because of the existence of universal extensions and
because   $\kappa^*_\mu(\K)<\mu^+$.  At successor stages we first take an extension of $(\bar M,\bar a,\bar N)^\beta$ indexed by $I_{\beta+1}$ and realizing all the strong types over the models in 
  $(\bar M,\bar a,\bar N)^\beta$.  This tower may not be reduced, but by Fact \ref{density of reduced}, it has a reduced extension.  At limit stages take unions of the chain of towers defined so far.  
  
  Notice that by Fact \ref{union of reduced is reduced}, the tower $\T^\delta$ formed by the union of all the $(\bar M,\bar a,\bar N)^\beta$ is reduced.  Furthermore, by Theorem \ref{reduced are continuous} every one of the reduced towers $\T^j$ is continuous at $\theta$ because $\cf(\theta)\geq\kappa^*_\mu(\K)$.  Therefore $M^\delta_{i_\theta}=\Union_{k<\theta}M^\delta_{i_k}$,
and by the definition of the ordering $<$ on towers, the last model in this tower ($M^\delta_{i_\theta}$) is a $(\mu,\delta)$-limit model witnessed by $\langle M^j_{i_\theta}\mid j<\delta\rangle$.  Since $M^1_{i_\theta}$ is universal over $M$, we have that $M^\delta_{i_\theta}$ is $(\mu, \delta)$-limit over $M$.

 Next to see that $M^\delta_{i_\theta}$ is also a $(\mu,\theta)$-limit model, notice that $\T^\delta$ is relatively full by condition \ref{realizing types} of the construction and the same argument as \cite[Claim 5.11]{gvv}.
Therefore by 
  Theorem \ref{reduced are continuous} and our choice of $\delta$ with $\cf(\delta)\geq\kappa^*_\mu(\K)$, the last model $M^\delta_{i_\theta}$ in this relatively full tower is a $(\mu,\theta)$-limit model over $M$.

This completes the proof of Theorem \ref{uniqueness theorem}.

\bibliographystyle{asl}
\bibliography{bounded-kappa-12_25_2022.bib}

\end{document}